\documentclass[a4paper,reqno]{amsart}

\usepackage{amssymb,upref}
\usepackage[mathcal]{euscript}

\newcommand{\bydef}{:=}

\newcommand{\id}{\mathrm{id}}

\newcommand{\espan}[1]{\mathrm{span}\left\{#1\right\}}

\DeclareMathOperator{\degree}{\mathrm{degree}}

\newcommand{\cA}{\mathcal{A}}
\newcommand{\cB}{\mathcal{B}}
\newcommand{\cC}{\mathcal{C}}

\newcommand{\cL}{\mathcal{L}}

\newcommand{\frg}{{\mathfrak g}}
\newcommand{\frh}{{\mathfrak h}}

\newcommand{\fri}{{\mathfrak i}}
\newcommand{\frgsp}{\mathfrak{gsp}}
\newcommand{\frpgsp}{\mathfrak{pgsp}}

\newcommand{\ZZ}{\mathbb{Z}}

\newcommand{\FF}{\mathbb{F}}

\DeclareMathOperator{\Hom}{\mathrm{Hom}}

\DeclareMathOperator{\Aut}{\mathrm{Aut}}
\DeclareMathOperator{\AAut}{\mathbf{Aut}}
\DeclareMathOperator{\Der}{\mathrm{Der}}
\DeclareMathOperator{\Skew}{\mathrm{Skew}}
\DeclareMathOperator{\supp}{\mathrm{Supp}\,}

\newcommand{\ad}{\mathrm{ad}}
\newcommand{\Ad}{\mathrm{Ad}}

\newcommand{\frsl}{{\mathfrak{sl}}}
\newcommand{\frsp}{{\mathfrak{sp}}}

\newcommand{\frpsl}{{\mathfrak{psl}}}
\newcommand{\frgl}{{\mathfrak{gl}}}

\def\hregleta{\hrule height .5pt}
\def\hreglon{\hrule height1pt}
\def\vreglon{\vrule height 12pt width1pt depth 4pt}
\def\vregleta{\vrule width .5pt}

\def\hregletafill{\leaders\hregleta\hfill}

\newtheorem{theorem}{Theorem}

\newtheorem{lemma}[theorem]{Lemma}
\newtheorem{corollary}[theorem]{Corollary}

\theoremstyle{definition}

\newtheorem{remark}[theorem]{Remark}

\newenvironment{romanenumerate}
 {\begin{enumerate}
 
 }{\end{enumerate}}

\begin{document}

\title[Some features of Cayley algebras, and $G_2$, in low characteristics]{Some special features of Cayley algebras, and $G_2$, in low characteristics}

\author[A.~Castillo-Ramirez]{Alonso Castillo-Ramirez${}^\star$}
\address{School of Engineering and Computing Sciences, Durham University, South Road, Durham, DH1 3LE, United Kingdom}
\email{alonso.castillo-ramirez@durham.ac.uk}
\thanks{${}^\star$ Supported by the 150th Anniversary Postdoctoral Mobility Grant (PMG14-15 01) of the London Mathematical Society}

\author[A.~Elduque]{Alberto Elduque${}^{**}$}
\address{Departamento de Matem\'{a}ticas
 e Instituto Universitario de Matem\'aticas y Aplicaciones,
 Universidad de Zaragoza, 50009 Zaragoza, Spain}
\email{elduque@unizar.es}
\thanks{${}^{**}$ Supported by the Spanish Ministerio de Econom\'{\i}a y Competitividad---Fondo Europeo de Desarrollo Regional (FEDER) MTM2013-45588-C3-2-P, and by the Diputaci\'on General de Arag\'on---Fondo Social Europeo (Grupo de Investigaci\'on de \'Algebra)}



\date{}

\begin{abstract}
Some features of Cayley algebras (or algebras of octonions) and their Lie algebras of derivations over fields of low characteristic are presented. More specifically, over fields of characteristic $7$, explicit embeddings of any twisted form of the Witt algebra into the simple split Lie algebra of type $G_2$ are given. Over fields of characteristic $3$, even though the Lie algebra of derivations of a Cayley algebra is not simple, it is shown that still two Cayley algebras are isomorphic if and only if their Lie algebras of derivations are isomorphic. Finally, over fields of characteristic $2$, it is shown that the Lie algebra of derivations of any Cayley algebra is always isomorphic to the projective special linear Lie algebra of degree four. The twisted forms of this latter algebra are described too.
\end{abstract}

\maketitle



\section{Introduction}

Let $\FF$ be an arbitrary field. Cayley algebras (or algebras of octonions) over $\FF$ constitute a well-known class of nonassociative algebras (see, e.g. \cite[Chapter VIII]{KMRT} and references therein). They are unital nonassociative algebras $\cC$ of dimension eight over $\FF$, endowed with a nonsingualr quadratic multiplicative form (the \emph{norm}) $q:\cC\rightarrow \FF$. Hence $q(xy)=q(x)q(y)$ for any $x,y\in\cC$, and the polar form $b_q(x,y)\bydef q(x+y)-q(x)-q(y)$ is a nondegenerate bilinear form.

Any element in a Cayley algebra $\cC$ satisfies the degree $2$ equation:
\begin{equation}\label{eq:CayleyHamilton}
x^2-b_q(x,1)x+q(x)1=0.
\end{equation}
Besides, the map $x\mapsto \bar x\bydef b_q(x,1)1-x$ is an involution (i.e., an antiautomorphism of order $2$) and the \emph{trace} $t(x)\bydef b_q(x,1)$ and norm $q(x)$ are given by $t(x)1=x+\bar x$, $q(x)1=x\bar x=\bar x x$ for any $x \in \cC$. Two Cayley algebras $\cC_1$ and $\cC_2$, with respective norms $q_1$ and $q_2$, are isomorphic if and only if the norms $q_1$ and $q_2$ are isometric.

If the characteristic of $\FF$ is not $2$, then $\cC=\FF 1\oplus \cC^0$, where $\cC^0$ is the subspace of trace zero elements (i.e., the subspace orthogonal to $\FF 1$ relative to $b_q$). For $x,y\in \cC^0$,  \eqref{eq:CayleyHamilton} shows that $xy+yx=-b_q(x,y)1$, while $t([x,y])=[x,y]+\overline{[x,y]}=[x,y]+[\bar y,\bar x]=0$, so $[x,y]\bydef xy-yx\in\cC^0$. In particular,
\begin{equation}\label{eq:xyC0}
xy=-\frac{1}{2}b_q(x,y)1+ \frac{1}{2}[x,y],
\end{equation}
so the projection of $xy$ in $\cC^0$ is $\frac{1}{2}[x,y]$. Moreover, the following relation holds (see, e.g. \cite[Theorem 4.23]{EKmon}):
\begin{equation}\label{eq:xyy}
[[x,y],y]=2b_q(x,y)y-2b_q(y,y)x,
\end{equation}
so the multiplication in $\cC$ and its norm are determined by the bracket in $\cC^0$.

The Lie algebra of derivations of a Cayley algebra $\cC$ is defined by
\[ \Der(\cC) \bydef \{ d\in \frgl(\cC): d(xy)=d(x)y+xd(y)\ \forall x,y\in\cC\}. \]
In general, if $M$ is a module for a Lie algebra $\cL$, we say that a bilinear product $\cdot : M \times M \rightarrow M$ is \emph{$\cL$-invariant} if $\cL$ acts on $(M, \cdot)$ by derivations: $x(u \cdot v) = x(u)\cdot v + u \cdot x(u)$, for any $x \in \cL$, $u,v \in M$.

If the norm $q$ of a Cayley algebra $\cC$ is isotropic (i.e., there exists $0\ne x\in \cC$ with $q(x)=0$), then $\cC$ is unique up to isomorphism. In this case, the Cayley algebra $\cC$ is said to be \emph{split}, and it has a \emph{good basis} $\{ p_1, p_2,u_1,u_2,u_3,v_1,v_2,v_3\}$ with multiplication given in Table \ref{ta:good_basis}. We denote by $\cC_s$ the split Cayley algebra.

\begin{table}[!h]\label{ta:good_basis}
\[
\vcenter{\offinterlineskip
\halign{\hfil$#$\enspace\hfil&#\vreglon
 &\hfil\enspace$#$\enspace\hfil
 &\hfil\enspace$#$\enspace\hfil&#\vregleta
 &\hfil\enspace$#$\enspace\hfil
 &\hfil\enspace$#$\enspace\hfil
 &\hfil\enspace$#$\enspace\hfil&#\vregleta
 &\hfil\enspace$#$\enspace\hfil
 &\hfil\enspace$#$\enspace\hfil
 &\hfil\enspace$#$\enspace\hfil&#\vreglon\cr
 &\omit\hfil\vrule width 1pt depth 4pt height 10pt
   &p_1&p_2&\omit&u_1&u_2&u_3&\omit&v_1&v_2&v_3&\cr
 \noalign{\hreglon}
 p_1&& p_1&0&&u_1&u_2&u_3&&0&0&0&\cr
 p_2&&0&p_2&&0&0&0&&v_1&v_2&v_3&\cr
 &\multispan{12}{\hregletafill}\cr
 u_1&&0&u_1&&0&v_3&-v_2&&-p_1&0&0&\cr
 u_2&&0&u_2&&-v_3&0&v_1&&0&-p_1&0&\cr
 u_3&&0&u_3&&v_2&-v_1&0&&0&0&-p_1&\cr
 &\multispan{12}{\hregletafill}\cr
 v_1&&v_1&0&&-p_2&0&0&&0&u_3&-u_2&\cr
 v_2&&v_2&0&&0&-p_2&0&&-u_3&0&u_1&\cr
 v_3&&v_3&0&&0&0&-p_2&&u_2&-u_1&0&\cr
 \noalign{\hreglon}}}
\]
\caption{{\vrule width 0pt height 15pt}Multiplication table in a good basis of the split Cayley algebra.}
\end{table}

Given a finite-dimensional simple Lie algebra $\frg$ of type $X_r$ over the complex numbers, and a Chevalley basis $\cB$, let $\frg_\ZZ$ be the $\ZZ$-span of $\cB$ (a Lie algebra over $\ZZ$). The Lie algebra $\frg_\FF\bydef\frg_\ZZ\otimes_\ZZ \FF$ is the \emph{Chevalley algebra} of type $X_r$. In particular, the Chevalley algebra of type $G_2$ is isomorphic to $\Der(\cC_s)$ (see, e.g. \cite[\S 4.4]{EKmon}). For any Cayley algebra $\cC$, the Lie algebra $\Der(\cC)$ is a twisted form of the Chevalley algebra $\Der(\cC_s)$. (Recall that, if $\cA$ and $\cB$ are algebras over $\FF$, then $\cA$ is a \emph{twisted form} of $\cB$ whenever $\cA\otimes_\FF \FF_{\text{alg}}\cong \cB\otimes_\FF\FF_{\text{alg}}$, for an algebraic closure $\FF_{\text{alg}}$ of $\FF$.)

If the characteristic of $\FF$ is neither $2$ nor $3$, then the Chevalley algebra of type $G_2$ is simple; this is the split simple Lie algebra of type $G_2$. Moreover, two Cayley algebras $\cC_1$ and $\cC_2$ are isomorphic if and only if their Lie algebras of derivations are isomorphic (see \cite[Theorem IV.4.1]{Seligman} or \cite[Theorem 4.35]{EKmon}).

The goal of this paper is to show some surprising features of Cayley algebras over fields of characteristic $7$, $3$ and $2$. 

Section \ref{se:char7} studies the case of characteristic $7$ and is divided in three subsections. In Section \ref{subsec1}, we review a construction of $\cC_s$ due to Dixmier \cite{Dixmier} in terms of transvectants which was originally done in characteristic $0$, but it is valid in any characteristic $p \geq 7$. In this construction, $\cC_s$ appears as the direct sum of the trivial one-dimensional module and the restricted irreducible seven-dimensional module $V_6$ for the simple Lie algebra $\frsl_2(\FF)$, which embeds into $\Der(\cC_s)$ as its principal $\frsl_2$ subalgebra. When the characteristic is $7$, this action of $\frsl_2(\FF)$ by derivations on $\cC_s$ may be naturally extended to an action by derivations of the Witt algebra $W_1\bydef \Der\left(\FF[X]/(X^7)\right)$, explaining the fact, first proved in \cite[Lemma 13]{Premet} (see also \cite{Herpel_Stewart}), that $W_1$ embeds into the split simple Lie algebra of type $G_2$. In Section \ref{subsec2}, we show that, when $\FF$ is algebraically closed, $V_6$ is the unique non-trivial non-adjoint restricted irreducible module for $W_1$ with a nonzero invariant product. Then, in Section \ref{subsec3} we prove that, in characteristic $7$ and even when the ground field is not algebraically closed, all the twisted forms of the Witt algebra embed into $\Der(\cC_s)$, and any two embeddings of the same twisted form are conjugate by an automorphism.
 
Section \ref{se:char3} is devoted to the case of characteristic $3$. In this situation, it is known that the Chevalley algebra of type $G_2$ is not simple, but it contains an ideal isomorphic to the projective special linear Lie algebra $\frpsl_3(\FF)$. We review this situation and prove that it is still valid that two Cayley algebras are isomorphic if and only if their Lie algebras of derivations are isomorphic.

Finally, in Section \ref{se:char2}, we prove that the Lie algebra of derivations $\Der(\cC)$ of any Cayley algebra $\cC$ over a field $\FF$ of characteristic $2$ is always isomorphic to the projective special linear Lie algebra $\frpsl_4(\FF)$. A proof of this fact when $\cC = \cC_s$ appears in \cite[Corollary 4.32]{EKmon}. Hence, in this case, it is plainly false that two Cayley algebras are isomorphic if and only if their Lie algebras of derivations are isomorphic. We show that the isomorphism classes of twisted forms of $\frpsl_4(\FF)$, which is here the Chevalley algebra of type $G_2$, are in bijection with the isomorphism classes of central simple associative algebras of degree $6$ endowed with a symplectic involution.

\medskip


\section{Characteristic $7$}\label{se:char7}


\subsection{Dixmier's construction} \label{subsec1}

Let $\FF[x,y]$ be the polynomial algebra in two variables over a field $\FF$ of characteristic $0$. The general linear Lie algebra $\frgl_2(\FF)$ acts by derivations on $\FF[x,y]$ preserving the degree of each polynomial. For any $n\geq 0$, denote by $V_n$ the subspace of homogeneous polynomials of degree $n$ in $\FF[x,y]$. For any $i,j,q\geq 0$ with $q\leq i,j$, consider the $q$-\emph{transvectant} $V_i\times V_j\rightarrow V_{i+j-2q}$ given by
\begin{multline*}
(f,g)_q\bydef \frac{(i-q)!}{i!}\frac{(j-q)!}{j!}
\left(\frac{\partial^qf}{\partial x^q}\frac{\partial^qg}{\partial y^q}
-\binom{q}{1}
\frac{\partial^qf}{\partial x^{q-1}\partial y}\frac{\partial^qg}{\partial x\partial y^{q-1}}\right.\\
\left. +\binom{q}{2}
\frac{\partial^qf}{\partial x^{q-2}\partial y^2}\frac{\partial^qg}{\partial x^2\partial y^{q-2}}
-+\cdots\right),
\end{multline*}
for $f\in V_i$ and $g\in V_j$.

It turns out that the split Cayley algebra $\cC_s$ is isomorphic to the algebra defined on $\FF 1\oplus V_6$ with multiplication given by
\begin{equation}\label{eq:CsDixmier}
(\alpha 1+f)(\beta 1+g)\bydef \bigl(\alpha\beta-\frac{1}{20}(f,g)_6\bigr)1+
\bigl(\alpha g+\beta f+(f,g)_3\bigr),
\end{equation}
for any $\alpha,\beta\in\FF$ and $f,g\in V_6$ (see \cite[3.6 Proposition]{Dixmier}); equipped with this product, the subspace $V_6$ becomes the subspace of trace zero elements in $\cC_s$. The existence of this isomorphism is based on the following identity given in \cite[3.5 Lemme]{Dixmier}:
\begin{equation}\label{eq:fgg}
((f,g)_3,g)_3=\frac{1}{20}\Bigl((f,g)_6g-(g,g)_6f\Bigr),
\end{equation}
for any $f,g\in V_6$.

Consider the following endomorphisms of $V_6$:
\[ e_{-1}:=\left.x\frac{\partial\ }{\partial y}\right|_{V_6},\quad
e_0:=-\frac{1}{2}\left.\left(x\frac{\partial\ }{\partial x}-y\frac{\partial\ }{\partial y}\right)\right|_{V_6},\quad e_1:=-\left.y\frac{\partial\ }{\partial x}\right|_{V_6}.\]
A direct calculation gives $[e_i,e_j]=(j-i)e_{i+j}$, for $i,j\in\{-1,0,1\}$, so these endomorphisms span a subalgebra of $\frgl(V_6)$ isomorphic to $\frsl_2(\FF)$ that acts by derivations on $V_6$. This construction also works when the characteristic of $\FF$ is $p\geq 7$, and, moreover, the map  $f\otimes g\mapsto (f,g)_3$ gives the only, up to scalars, linear map $V_6\otimes V_6\rightarrow V_6$ invariant under the action of $\frsl_2(\FF)$.

Now assume that the characteristic of $\FF$ is $7$. First, we will find a simpler formula describing the $3$-transvectant $(\; , \;)_3 : V_6 \times V_6 \rightarrow V_6$. For $0\leq i\leq 6$, denote
\[ m_i\bydef x^{6-i}y^i \in V_6.\]
Taking the indices modulo $7$, the action of $\frsl_2(\FF)$ on $V_6$ is given by the following formulas:
\begin{equation}\label{eq:e01-1mi}
\begin{split}
e_{-1}(m_i)&=im_{i-1} \text{ for } 1 \leq i \leq 6, \text{ and } e_{-1}(m_0) = 0, \\
e_0(m_i)&=(i+4)m_i \text{ for any } 0 \leq i \leq 6, \\
e_1(m_i)&=(i+1)m_{i+1} \text{ for } 0 \leq i \leq 5, \text{ and } e_1(m_6) = 0.
\end{split}
\end{equation}

For $i,j$ in the prime subfield $\FF_7$ of $\FF$, define the element $c(i,j)$ by
\begin{equation}\label{eq:cij}
c(i,j)=2(j-i)(4i+j-1)(4j+i-1)\ \bigl(\in\FF_7\subseteq \FF\bigr).
\end{equation}
It is clear that $c(i,j)=-c(j,i)$ for any $i,j \in \FF_7$, and $c(0,6)=1$. A straightforward computation gives, for any $i,j,k\in\FF_7$:
\begin{equation}\label{eq:cij_recurs}
(i+j+4k+1)c(i,j)=(i+4k+4)c(i+k,j)+(j+4k+4)c(i,j+k).
\end{equation}
Therefore, defining 
\begin{equation}\label{eq:mimj}
m_i\cdot m_j :=c(i,j)m_{i+j-3}
\end{equation}
(with indices modulo $7$), we have
\[
e_k(m_i\cdot m_j)=e_k(m_i)\cdot m_j+m_i\cdot e_k(m_j),
\]
for any $k \in \{-1,0,1\}$, $0\leq i,j\leq 6$. This means that the product in \eqref{eq:mimj} is $\frsl_2(\FF)$-invariant; hence, by uniqueness and since $(m_0,m_6)_3=m_3=m_0\cdot m_6$, we conclude that
\begin{equation}\label{eq:mimj_3}
m_i\cdot m_j=(m_i,m_j)_3
\end{equation}
for any $0 \leq i,j \leq 6$. The multiplication table of $V_6$ with this product is given in Table \ref{ta:multiplication_mis}.

\begin{table}[!h]
\setlength{\tabcolsep}{6pt}
\renewcommand{\arraystretch}{1.3}
\centering
\begin{tabular}{c|ccccccc|}
$\cdot$ & $m_0$ & $m_1$ & $m_2$ & $m_3$ & $m_4$ & $m_5$ & $m_6$ \\ \hline 
$m_0$ & $0$ & $0$ & $0$ & $-m_0$ & $3m_1$ & $-3m_2$ & $m_3$   \\
$m_1$ & $0$ & $0$ & $3m_0$ & $m_1$ & $0$ & $-m_3$ & $-3m_4$ \\
$m_2$ & $0$ & $-3m_0$ & $0$ & $m_2$ & $-m_3$ & $0$ & $3m_5$\\
$m_3$ & $m_0$ & $-m_1$ & $-m_2$ & $0$ & $m_4$ & $m_5$ & $-m_6$ \\
$m_4$ & $-3m_1$ & $0$ & $m_3$ & $-m_4$ & $0$ & $3m_6$ & $0$ \\
$m_5$ & $3m_2$ & $m_3$ & $0$ & $-m_5$ & $-3m_6$ & $0$ & $0$\\
$m_6$ & $-m_3$ & $3m_4$ & $-3m_5$ & $m_6$ & $0$ & $0$ & $0$ \\ \hline
\end{tabular}
\caption{{\vrule width 0pt height 15pt}Multiplication table of $(V_6,\cdot)$.}
\label{ta:multiplication_mis}
\end{table} 

\begin{remark}\label{re:explicit_iso}
The above arguments show that the split Cayley algebra $\cC_s$ is isomorphic to the algebra defined on $\FF 1\oplus V_6$ with multiplication given by \eqref{eq:CsDixmier}, or equivalently, by
\[
(\alpha 1+f)(\beta 1+g)\bydef \bigl(\alpha\beta+(f,g)_6\bigr)1+
\bigl(\alpha g+\beta f+f\cdot g\bigr).
\]
An explicit isomorphism between $\FF 1\oplus V_6$ and $\cC_s$, in terms of a good basis of  $\cC_s$ as in Table \ref{ta:good_basis}, is given by:
\[
\left( \begin{matrix}
\ 1 \ & \ m_0\ &\ m_1\ &\ m_2\ &\ m_3\ &\ m_4\ &\ m_5\ &\ m_6\ \\
\downarrow&\downarrow&\downarrow&\downarrow&\downarrow&\downarrow&\downarrow&\downarrow\\
p_1 + p_2 & -3v_3&3u_2&3u_1&-p_1+p_2&-3v_1&-3v_2&3u_3
\end{matrix} \right).
\]
\end{remark}

Now, for any $k \in \{ 2,\ldots,5 \}$, define the endomorphism $e_k$ of $V_6$ by
\begin{equation}\label{eq:eks}
e_k(m_i)=\begin{cases} (i+4k+4)m_{i+k}&\text{if $i+k \leq 6$,}\\
   0&\text{otherwise.}
   \end{cases}
\end{equation}
We will show that $e_k$ is a derivation of $(V_6,\cdot)$ for any $k=\{-1,0,\ldots,5\}$. Take $0\leq i,j \leq 6$. If $i+j+k-3\leq 6$, $i+k\leq 6$, and $j+k\leq 6$, then
\begin{equation}\label{eq:ekmimj}
\begin{split}
e_k (m_i \cdot m_j) &= (i+j + 4k + 1) c(i,j)  m_{i+j+k-3} \\
 &= \Bigl((i+4k+4)c(i+k,j) + (j+4k+4)c(i,j+k)\Bigr) m_{i+j+k-3} \\
 &= e_k(m_i) \cdot m_j + m_i \cdot e_k(m_j).
\end{split}
\end{equation}
If $i+j+k-3>6$, then $e_k(m_i\cdot m_j)=0=e_k(m_i)\cdot m_j=m_i\cdot e_k(m_j)$. Finally, if $i+j+k-3\leq 6$ and $i+k>6$ (the same happens if $j+k>6$), then $e_k(m_i)=0$ and $c(i+k,j)$ is one of $c(0,0)$, $c(0,1)$, $c(0,2)$ or $c(1,1)$, but all these are equal to $0$, so \eqref{eq:ekmimj} applies.

Because of \eqref{eq:CsDixmier} and \eqref{eq:fgg}, we may extend the action of $e_k$ to $ \FF 1\oplus V_6 \cong \cC_s$ by means of $e_k(1)=0$, obtaining that $e_k$ is a derivation of the split Cayley algebra $\cC_s$. Furthermore, one checks at once that
\begin{equation}\label{eq:Witt_basis}
[e_i,e_j]=(j-i)e_{i+j}
\end{equation}
for any $i,j \in \{ -1,0,\ldots,5 \}$, with $e_i=0$ when $i$ is outside $\{ -1,0,\ldots,5 \}$. Thus, the span of $\{ e_i:-1\leq i\leq 5 \}$ in $\Der(V_6,\cdot)$ is isomorphic to the Witt algebra $W_1:=\Der\left(\FF[X]/(X^7)\right)$; the map $e_i\leftrightarrow x^{i+1}\frac{\partial }{\partial x}$ gives an explicit isomorphism of these Lie algebras, where $x$ denotes the class of $X$ modulo $(X^7)$.

Since $\Der(\cC_s)$ is the split simple Lie algebra of type $G_2$, the above arguments provide an elementary proof of the next result.

\begin{theorem}[\cite{Premet}]\label{th:W1_char7}
If the characteristic of the ground field $\FF$ is $7$, then the Witt algebra $W_1$ embeds as a subalgebra of the simple split Lie algebra of type $G_2$.
\end{theorem}

\begin{remark}
A specific embedding of $W_1$ into $\Der(\cC_s)$ may be given in terms of a good basis of $\cC_s$. Given $x,y\in\cC_s$, the derivation of $\cC_s$ defined by $D_{x,y}:=\ad_{[x,y]}+[\ad_x,\ad_y]$ (where $\ad_x(y):=[x,y]$) is called the \emph{inner derivation induced by $x$ and $y$}. Then, the assignment
\begin{align*}
e_{-1} &\mapsto D_{p_1 - p_2,u_1} + D_{u_2, v_1},& e_3 &\mapsto -D_{v_1, v_2} = 4D_{p_1 - p_2,u_3}, \\
e_0 &\mapsto 2D_{u_3, v_3} + 3D_{u_2,v_2},  &\quad e_4 &\mapsto 3D_{v_1, u_3}, \\
e_1 &\mapsto D_{p_1 - p_2,v_1} + D_{u_1,v_2}, & e_5 &\mapsto 5D_{v_2,u_3},  \\
e_2 &\mapsto 3D_{u_1, u_3} = 2D_{p_1 - p_2,v_2}, &&
\end{align*} 
gives an explicit embedding of $W_1$ into $\Der(\cC_s)$.
\end{remark}

\smallskip


\subsection{Invariant bilinear products on modules for $W_1$} \label{subsec2}

It turns out that $V_6$ is a very special module for the Witt algebra in characteristic $7$ (see Theorem \ref{th:main} below).

Assume for this section that $\FF$ is an algebraically closed field of characteristic $p \geq 5$. Let $\cL$ be a finite-dimensional restricted Lie algebra over $\FF$ with $p$-mapping denoted by $[p]$. Given any character $\chi \in \cL^*$, there is a finite-dimensional algebra $u(\cL,\chi)$, called the \emph{reduced enveloping algebra of $\cL$ associated with $\chi$}, that is a quotient of the universal enveloping algebra of $\cL$ and whose irreducible modules coincide precisely with the irreducible modules for $\cL$ with character $\chi$. We say that $V$ is a \emph{restricted module} for $\cL$ if there is a representation $\rho : \cL \rightarrow \frgl(V)$ that is a morphism of restricted Lie algebras, i.e., $\rho(x^{[p]}) = \rho(x)^{p}$, for any $x \in \cL$. When $\chi=0$, it turns out that the irreducible modules for $u(\cL, \chi)$ coincide precisely with the restricted irreducible modules for $\cL$.

Let $W=W_1$ be the Witt algebra over $\FF$ with basis $\{e_i:-1\leq i\leq p-2\}$ and Lie bracket as in \eqref{eq:Witt_basis}. It is well known that $W$ has a $p$-mapping given by $e_i^{[p]} := \delta_{i,0}e_i$.

For each $i \in \{-1,0,...,p-2\}$, define $W_{(i)} = \langle e_i, ..., e_{p-2} \rangle$. Consider the $p$-dimensional Verma modules $L(\lambda) = u(W,0) \otimes_{u(W_{(0)},0)} k_\lambda$, where $\lambda \in \{0,1,...,p-1\}$, and $k_\lambda$ is the one-dimensional module for $W_{(0)}$ on which $W_{(1)}$ acts trivially and $e_0$ acts by multiplication by $\lambda$. By \cite[Lemma 2.2.1]{N92}), $L(\lambda)$ has a basis $\left\{ m_{0},m_{1},...,m_{p-1}\right\}$ on which the action of $W$ is given by
\[ e_{k}(m_{j}) = ( j + ( \lambda +1 ) ( k+1 )) m_{k+j}, \]
where $m_{j}=0$ for $j$ outside $\{ 0,...,p-1\}$.

It was established in \cite{C41} (see also \cite{S77,FN98}) that any restricted irreducible module for $W$ is isomorphic to one of the following: 
\begin{enumerate}
\item The trivial one-dimensional module.
\item The ($p-1$)-dimensional quotient $L(p-1)/\langle m_0 \rangle$.
\item The $p$-dimensional Verma module $L(\lambda)$, with $\lambda \in \{1,...,p-2 \}$.
\end{enumerate}

In the next result the invariant bilinear products $L\times L \rightarrow L$ (not to be confused with invariant bilinear forms!) on irreducible restricted modules for the Witt algebra are determined.

\begin{theorem}\label{th:main}
Let $W=W_1$ be the Witt algebra over an algebraically closed field $\FF$ of characteristic $p \geq 5$. Then:
\begin{enumerate}
\item If $p \neq 7$, there is no non-trivial non-adjoint restricted irreducible module for $W$ with a nonzero invariant bilinear product.
\item If $p = 7$, there is a unique non-trivial non-adjoint restricted irreducible module for $W$ with a nonzero invariant bilinear product. Up to isomorphism and scaling of the product, this unique module is $V_6$ with product given in \eqref{eq:mimj}. ($V_6$ is a module for $W$ by means of \eqref{eq:e01-1mi} and \eqref{eq:eks}.)   
\end{enumerate}
\end{theorem}
\begin{proof}
We will use the above classification of restricted irreducible modules for $W$ and follow several steps:
\begin{romanenumerate}
\item 
Let $\lambda \in \{0,1,...,p-1\}$ and suppose that $L(\lambda)$ is equipped with $\cdot$, a $W$-invariant bilinear product. We claim that $m_{0}\cdot m_{p-1}= \mu m_{\lambda }$, for some scalar $\mu \in \FF$ that determines the invariant product. Indeed, first observe that 
\[\begin{split}
e_{0}\left( m_{0}\cdot m_{p-1}\right)  &=\left( \lambda +1\right) m_{0}\cdot m_{p-1}+\left( p-1+\left( \lambda +1\right) \right) m_{0}\cdot m_{p-1} \\
&=\left( 2\lambda +1\right) m_{0}\cdot m_{p-1}.
\end{split}\]
Therefore, $m_{0}\cdot m_{p-1}$ is an eigenvector of the action of $e_{0}$ with eigenvalue $2\lambda +1$. As $m_{\lambda }$ is also an eigenvector of the action of $e_{0}$ with eigenvalue $2 \lambda +1$, and the action of $e_{0}$ is diagonal with $p$ distinct eigenvalues, we must have that $m_{0}\cdot m_{p-1}$ is a scalar multiple of $m_{\lambda}$. Furthermore, $m_0\otimes m_{p-1}$ generates the $W$-module $L(\lambda)\otimes L(\lambda)$, so this scalar determines the invariant product.

\item For any $p \geq 5$, we will show that if $\cdot$ is a $W$-invariant bilinear product on the irreducible module $L(p-1) / \langle m_0 \rangle$, then it must be the zero product. Denote by $\bar x$ the class of an element $x\in L(p-1)$ modulo $\langle m_0\rangle$; then, 
\[ e_0 \left( \overline{m}_1 \cdot \overline{m}_{p-1} \right) = \overline{m}_1 \cdot \overline{m}_{p-1} + (p-1) \overline{m}_1 \cdot \overline{m}_{p-1} =\overline{0}. \]
As all the eigenvalues of the action of $e_0$ on $L(p-1) / \langle m_0 \rangle$ are nonzero, this implies that $\overline{m}_1 \cdot \overline{m}_{p-1}= \overline{0}$. Now, using the $W$-invariance, it is easy to show that $\cdot$ must be the zero product.

\item As $L(p-2)$ is the adjoint module for $W$, it obviously has a $W$-invariant bilinear product. Hence, we exclude this case from further observations.  

\item Let $p\geq 5$ and $\lambda \in \left\{ 1,...,p-3\right\}$. We will show that if $L\left( \lambda \right)$ has a nonzero $W$-invariant bilinear product $\cdot$, then $p=7$ and $\lambda =3$.

By step (i), we may assume that $m_0\cdot m_{p-1}=m_\lambda$. For $k \in \{1,...,p-2 \}$,
\begin{eqnarray*}
(\lambda +1)(k+1)m_k\cdot m_{p-1} & = &e_k(m_0)\cdot m_{p-1} \\
 & = & e_k(m_0\cdot m_{p-1})= e_k(m_\lambda) \\
 & = & (\lambda+(\lambda+1)(k+1))m_{\lambda +k},
\end{eqnarray*}
so
\begin{equation}\label{eq:mkmp-1}
m_{k}\cdot m_{p-1}=\frac{\left( \lambda +1\right) \left( k+1\right) +\lambda}{\left( \lambda +1\right) \left( k+1\right) }m_{\lambda+k },
\end{equation}
for $k\in \left\{ 1,...,p-2\right\}$. On the other hand, $e_1(e_1(m_\lambda))=e_1(e_1(m_0)\cdot m_{p-1}$, and this gives
\begin{equation}\label{eq:m2mp-1}
3\left( 3\lambda +2\right) \left( \lambda +1\right) m_{\lambda +2} =2\left( \lambda +1\right) \left( 2\lambda +3\right) m_{2}\cdot m_{p-1}. 
\end{equation}

If $2\lambda +3=0$, \eqref{eq:m2mp-1} gives $3\lambda+2=0$, and this implies $p=5$ and $\lambda=1$. But then, using \eqref{eq:mkmp-1} we get $e_1(m_2\cdot m_4)=2e_1(m_3)=4m_4$, while at the same time  
\[ e_{1}\left( m_{2}\cdot m_{4}\right) = (e_1m_3) \cdot m_4 + m_3 \cdot (e_1 m_4) = m_{3}\cdot m_{4} = -2m_{4},\]
which is a contradiction. 

Hence, we assume for the rest of the proof $2\lambda +3\ne 0$. By \eqref{eq:m2mp-1}, we have
\[
 m_{2}\cdot m_{6}=\frac{3\left( 3\lambda +2\right) }{2\left( 2\lambda +3\right) }m_{\lambda +2}.
  \]
Comparing this with \eqref{eq:mkmp-1} with $k=2$ we obtain
\[
 \frac{3\left( 3\lambda +2\right) }{2\left( 2\lambda +3\right) } = \frac{4\lambda +3}{3\left( \lambda +1\right) }, \text{ so } \lambda \left( 11\lambda +9\right) =0.
\]
If $p=11$, the above relation implies that $\lambda =0$, which is a contradiction with our choice of $\lambda $, so no invariant bilinear product exists for $p=11$.

For the rest of the proof, assume that $p\neq 11$ and hence $
\lambda =-\frac{9}{11}$.
Now, from $e_{1}\left( e_{1}\left( e_{1}(m_{\lambda} )\right) \right) =e_{1}\left(e_{1}\left( e_{1}(m_{0})\cdot m_{p-1}\right) \right)$, we obtain
\[ 
\left( 3\lambda +2\right) \left( 3\lambda +3\right) \left( 3\lambda +4\right) m_{\lambda +3} = \left( 2\lambda +2\right)\left( 2\lambda +3\right) \left( 2\lambda +4\right) m_{3}\cdot m_{p-1}.
\]
Thus,
\[ 
m_{3}\cdot m_{p-1}=\frac{3\left( 3\lambda +2\right) \left( 3\lambda +4\right) }{4\left( 2\lambda +3\right) \left( \lambda +2\right) }m_{\lambda +3}.
\]
Comparing this with \eqref{eq:mkmp-1} for $k=3$ we obtain
\[ 
\frac{3\left( 3\lambda +2\right) \left( 3\lambda +4\right) }{\left( 2\lambda +3\right) \left( \lambda +2\right) } = \frac{5\lambda +4}{\left( \lambda +1\right) }. 
\]
Hence,
\[ 
17\lambda ^{2}+38\lambda +20=0. 
\]
Substituting $\lambda =-\frac{9}{11}$ in this relation we obtain that $p\mid 35$, so either $p=7$ or $p=5$. If $p=5$, then $\lambda =-\frac{9}{11}=1$, and it was shown above that no nonzero invariant bilinear product exists in this case. If $p=7$, then $\lambda =-\frac{9}{11}=3$. This completes the proof. \qedhere
\end{romanenumerate}
\end{proof}

\smallskip

\begin{remark} \label{re:abs irr}
Let $V$ and $U$ be two irreducible modules for a Lie algebra $\cL$ over an arbitrary field $\FF$, and let $\FF_{\text{alg}}$ be an algebraic closure of $\FF$. Suppose that $V\otimes_\FF\FF_{\text{alg}}$ and $U\otimes_\FF\FF_{\text{alg}}$ are isomorphic as modules for $\cL\otimes_\FF \FF_{\text{alg}}$; then, $V$ and $U$ are isomorphic as modules for $\cL$. Indeed, if $V\otimes_\FF\FF_{\text{alg}}$ and $U\otimes_\FF\FF_{\text{alg}}$ are isomorphic, then $\Hom_{\cL}(V,U)\otimes_\FF\FF_{\text{alg}}\cong \Hom_{\cL\otimes_\FF\FF_{\text{alg}}}(V\otimes_\FF\FF_{\text{alg}},U\otimes_\FF\FF_{\text{alg}})\ne 0$, so $\Hom_\cL(V,U)\ne 0$. The result follows since, by irreducibility, any nonzero $\cL$-module homomorphism from $V$ to $U$ is an $\cL$-module isomorphism. In particular, this implies that, even when the ground field is not algebraically closed, Theorem \ref{th:main} applies to modules for the Witt algebra that are absolutely irreducible (i.e., they remain irreducible after extending scalars to an algebraic closure).
\end{remark}

\smallskip


\subsection{Embeddings of $W_1$ and its twisted forms into $G_2$} \label{subsec3}

It is shown in \cite{Premet,Herpel_Stewart} that, over an algebraically closed field of characteristic $7$, the simple Lie algebra of type $G_2$ contains a unique conjugacy class of subalgebras isomorphic to the Witt algebra. With our results above, a different proof may be given, valid for not necessarily algebraically closed fields.

\begin{theorem}\label{th:conjugation}
Let $\FF$ be an arbitrary field of characteristic $7$. Then any two subalgebras $S_1$ and $S_2$ of $\Der(\cC_s)$ isomorphic to the Witt algebra are conjugate: there is an automorphism $\varphi$ of $\cC_s$ such that $S_2=\varphi S_1\varphi^{-1}$.
\end{theorem}
\begin{proof}
Let $S:=\espan{e_i:-1\leq i\leq 5}$ be the subalgebra of $\frg:=\Der(\cC_s)$ isomorphic to the Witt algebra given by equations \eqref{eq:e01-1mi} and \eqref{eq:eks}, and let $\tilde S$ be an arbitrary subalgebra of $\frg$ isomorphic to the Witt algebra. Take a basis $\{\tilde e_i:-1\leq i\leq 5\}$ of $\tilde S$ with $[\tilde e_i,\tilde e_j]=(j-i)\tilde e_{i+j}$. We will follow several steps:

\begin{romanenumerate}
\item $\tilde S$ is a maximal subalgebra of $\frg$.
\begin{proof}
Since the Witt algebra does not admit nonsingular invariant bilinear forms (see, e.g. \cite[Theorem 4.2]{F86}), the restriction of the Killing form $\kappa$ of $\frg$ to $\tilde S$ is trivial; hence, by dimension count, $\frg/\tilde S$ is isomorphic, as a module for $\tilde S$, to the dual of the adjoint module for $\tilde S$. In particular, $\frg/\tilde S$ is an irreducible module for $\tilde S$, and this shows the maximality of $\tilde S$.
\end{proof}

\item The representation of $\tilde S$ on $\cC_s^0$, the subspace of trace zero elements of $\cC_s$, is restricted and absolutely irreducible.
\begin{proof}
As $\tilde S$ is an ideal in its $p$-closure in $\frg$, step (i) implies that $\tilde S$ is a restricted subalgebra of $\frg$, and hence, the corresponding representation of $\tilde S$ on $\cC_s^0$ is restricted. 

In order to prove the second part, we may assume that $\FF$ is algebraically closed. Recall that the possible dimensions of restricted irreducible representations of the Witt algebra are $1$, $6$ and $7$. Suppose that $\cC_s^0$ has a one-dimensional trivial $\tilde S$-submodule $X$. The space $X$ may be either nondegenerate or totally isotropic with respect to the symmetric bilinear form $b_q$ of $\cC_s$. 

If $X = \FF x$ is nondegenerate, then $q(x) \neq 0$ and $b_q(x,1) = 0$. By \eqref{eq:CayleyHamilton}, $\FF 1 \oplus X$ is a two-dimensional composition subalgebra of $\cC_s$, so it is isomorphic to $\FF p_1\oplus\FF p_2$ (with $p_i$ as in Table \ref{ta:good_basis}). Now, by \cite[Corollary 1.7.3]{SV00}, the isomorphism $\FF 1 \oplus \FF x  \cong \FF p_1 \oplus \FF p_2$ may be extended to an automorphism of $\cC_s$. As $\tilde S$ annihilates $\FF 1 \oplus \FF x$, and the subalgebra of the derivations that annihilate $\FF p_1\oplus\FF p_2$ is isomorphic to $\frsl_3(\FF)$ (\cite[Proposition 4.29]{EKmon}), we obtain that the Witt algebra embeds into $\frsl_3(\FF)$, which is impossible. 

If $X$ is totally isotropic, then $X \subseteq X^{\perp}$, and $X^{\perp} / X$ is a $5$-dimensional module for $\tilde S$. As this cannot be irreducible, we deduce that all the composition factors of $\cC_s^0$ are one-dimensional. Hence, the representation $\rho$ of $\tilde S$ on $\cC_s^0$ is nilpotent. Since $\ad_{(\tilde e_0 + \tilde e_i)}$ is diagonalizable for $i\ne 0$, $-1 \leq i \leq 5$, then $\ad_{(\tilde e_0 + \tilde e_i)}^p = \ad_{(\tilde e_0+ \tilde e_i)}$. As the representation of $\tilde S$ on $\cC_s^0$ is restricted, then $\rho(\tilde e_0 + \tilde e_i)^p = \rho(\tilde e_0 + \tilde e_i)$, and the nilpotency implies that $\rho(\tilde e_0+\tilde e_i) =0$, for any $i\ne 0$, $-1 \leq i \leq 5$. This is a contradiction.

Finally, if $Y$ is a $6$-dimensional irreducible $\tilde S$-submodule for $\cC_s^0$, then it must be nondegenerate with respect to $b_q$. This implies that $Y^{\perp}$ is a one-dimensional nondegenerate submodule, and we may use the above argument with $X=Y^{\perp}$. 
\end{proof}

\item Note that $\cC_s^0$ is not the adjoint module for $\tilde S$ because of the existence of the invariant bilinear form $b_q$ on $\cC_s^0$. 

\item The previous steps together with Theorem \ref{th:main} and Remark \ref{re:abs irr} imply that, even when $\FF$ is not algebraically closed, there is a unique possibility, up to isomorphism, for $\cC_s^0$ as a module for $\tilde S$, and a unique, up to scalars, nonzero $\tilde S$-invariant product on $\cC_s^0$. Therefore, there is a basis $\{\tilde m_i: 0\leq i\leq 6\}$ of $\cC_s^0$ with 
\[
\tilde e_k(\tilde m_i)=\begin{cases} (i+4k+4)\tilde m_{i+k}&\text{if $i+k\geq 6$,}\\
   0&\text{otherwise,}
   \end{cases}
\]
and 
\[
\frac{1}{2}[\tilde m_i\tilde m_j]=c(i,j)\tilde m_{i+j-3}
\]
for $-1\leq k\leq 5$ and $0\leq i,j\leq 6$. Since the multiplication on a Cayley algebra is determined by the bracket of trace zero elements, the linear map $\varphi$ that takes $1$ to $1$ and $m_i$ to $\tilde m_i$, for $i \in \{ 0,\ldots,6 \}$, is an automorphism of $\cC_s$ such that $\tilde S=\varphi S\varphi^{-1}$.

\end{romanenumerate}
\end{proof}

\smallskip

To finish this section, note that the Witt algebra $W$ over a field $\FF$ of characteristic $7$ is equal to the Lie algebra $\FF[X]/(X^7)=\FF[Z]/(Z^7-1)$, where $Z=X+1$. Denote by $z$ the class of $Z$ modulo $(Z^7-1)=((Z-1)^7)$. The elements $f_i=z^{i+1}\frac{\partial\ }{\partial z}$, for $i \in \{-1,0,\ldots 5\}$, form a basis of $W$ with
\begin{equation}\label{eq:secondWitt_basis}
[f_i,f_j]=(j-i)f_{i+j},
\end{equation}
where, contrary to \eqref{eq:Witt_basis}, the indices are taken modulo $7$. Now we may define an action of $W$ on $V_6$ by changing slightly the definition in \eqref{eq:eks}:
\begin{equation}\label{eq:fks}
f_k(m_i)=(i+4k+4)m_{i+k},
\end{equation}
for $-1\leq k\leq 5$ and $0\leq i\leq 6$, with indices taken modulo $7$. This gives another representation of $W$ on $V_6$:
\[
\begin{split}
f_r(f_s(m_i))&-f_s(f_r(m_i))\\
 &=\Bigl((i+s+4r+4)(i+4s+4)-(i+r+4s+4)(i+4r+4)\Bigr)m_{i+s+r}\\
 &=(s-r)(i+4(r+s)+4)m_{i+s+r}\\
 &=[f_r,f_s](m_i).
\end{split}
\]
Equation \eqref{eq:cij_recurs} proves that this is a representation by derivations, so $W$ embeds in $\Der(V_6,\cdot)$, and hence on $\Der(\cC_s)$ as well. This embedding is different from the one obtained through \eqref{eq:eks}, but Theorem \ref{th:conjugation} shows that they are conjugate.

We may even go a step further. For any $0\ne \alpha\in\FF$, consider the Lie algebra $W^\alpha=\Der\left(\FF[Y]/(Y^7-\alpha)\right)$, and denote by $y$ the class of $Y$ modulo $(Y^7-\alpha)$. For any natural number $i$, consider the element $\tilde f_i=y^{i+1}\frac{\partial\ }{\partial y}$ in $W^\alpha$, so $\{\tilde f_i:-1\leq i\leq 5\}$ is a basis of $W^\alpha$, and $f_{i+7}=\alpha f_i$ for any $i$. Define $m_{i+7}\bydef \alpha m_i$. For $0\leq i,j\leq 6$, $c(i,j)=0$ if $i+j-3>6$ or $i+j-3<0$, so if we modify \eqref{eq:fks} as follows:
\begin{equation}\label{eq:ftildeks}
\tilde f_k(m_i)=(i+4k+4)m_{i+j-3},
\end{equation}
for any $i,k$ (but now $\tilde f_{k+7}=\alpha\tilde f_k$ and $m_{i+7}=\alpha m_i$). The same computations as above show that $W^\alpha$ embeds in $\Der(\cC_s)$.

The twisted forms of the Witt algebra are precisely the algebras $W^\alpha$; if $0\ne\alpha\in\FF^7$, $W^\alpha$ is isomorphic to the Witt algebra, while if $\alpha\in\FF\setminus\FF^7$, $W^\alpha$ is the Lie algebra of derivations of the purely inseparable field extension $\FF[Y]/(Y^7-\alpha)$. Two algebras $W^\alpha$ and $W^\beta$ are isomorphic if and only if so are the algebras $\FF[Y]/(Y^7-\alpha)$ and $\FF[Y]/(Y^7-\beta)$. (See \cite{Allen_Sweedler69} or \cite{Waterhouse71}.)

Therefore, Theorems \ref{th:W1_char7} and \ref{th:conjugation} may be extended as follows:

\begin{theorem}\label{th:Walpha}
Over a field of characteristic $7$, all the twisted forms of the Witt algebra embed in the Lie algebra of derivations of the split Cayley algebra. Moreover, any two embeddings of the same twisted form of the Witt algebra in $\Der(\cC_s)$ are conjugate.
\end{theorem}

The last part of this theorem follows by the same arguments as in the proof of Theorem \ref{th:conjugation}.

\begin{remark}
If $\cC$ is a non split Cayley algebra (and hence it is a division algebra, since $q$ is anisotropic), then $\Der(\cC)$ contains no nonzero nilpotent derivation. Indeed, if $d\in\Der(\cC)$ is nilpotent, then $\ker d\vert_{\cC^0}$ is a subspace of $\cC^0$ and, since $q$ is anisotropic, $\cC^0=\ker d\vert_{\cC^0}\oplus \left(\ker d\vert_{\cC^0}\right)^\perp$. Besides, since $q$ is invariant under the action of $d$ (because of \eqref{eq:CayleyHamilton}), $d$ leaves $\left(\ker d\vert_{\cC^0}\right)^\perp$ invariant; this is a contradiction because the nilpotency of $d$ implies that any nonzero invariant subspace has nontrivial intersection with the kernel.

Therefore, $\Der(\cC)$ cannot contain subalgebras isomorphic to twisted forms of the Witt algebra, because these algebras contain nilpotent elements. (Recall that the action of the Witt algebra on $\cC_s^0$ is restricted, and hence so is the action on $\cC^0$ of any subalgebra of $\Der(\cC)$ isomorphic to a twisted form of the Witt algebra.)
\end{remark}

\medskip


\section{Characteristic $3$}\label{se:char3}

Let $\cC$ be a Cayley algebra over a field $\FF$ of characteristic $p \neq 2$. Then, the subspace of trace zero elements $\cC^0$ is closed under the commutator $[\,, \,]$, and it satisfies \eqref{eq:xyy}. The anticommutative algebra $\left( \cC^0, [\, , \, ] \right)$ is a central simple Malcev algebra. If $p \ne 2,3$, any central simple non Lie Malcev algebra is isomorphic to one of these. However, if $p=3$, then $\cC^0$ is a simple Lie algebra; more precisely, it is a twisted form of the projective special linear Lie algebra $\frpsl_3(\FF)$, and any such twisted form is obtained, up to isomorphism, in this way. (See \cite{AEMN} or \cite[Theorem 4.26]{EKmon}.)

Denote by $\AAut(\cA)$ the affine group scheme of automorphisms of a finite-dimensional algebra $\cA$. Equations \eqref{eq:xyC0} and \eqref{eq:xyy} show that the restriction map
\begin{equation}\label{eq:AutCAutC0}
\begin{split}
\AAut(\cC)&\longrightarrow \AAut(\cC^0)\\
f\ &\mapsto\quad f\vert_{\cC^0},
\end{split}
\end{equation}
gives an isomorphism of group schemes. (The reader may consult \cite[Chapter VI]{KMRT} for the basic facts of affine group schemes.)

For the rest of this section, assume that $\cC$ is a Cayley algebra over a field $\FF$ of characteristic $3$, and write $\frg\ :=\Der(\cC)$. Any derivation $d\in\frg$ satisfies $d(1)=0$ and $d(\cC^0)\subseteq \cC^0$, and hence, because of \eqref{eq:xyC0} and \eqref{eq:xyy}, we may identify $\frg$ with $\Der(\cC^0)$. Since $\cC^0$ is a Lie algebra, $\ad_{\cC^0}$ is an ideal of $\frg$: the ideal of inner derivations. In fact, $\ad_{\cC^0}$ is the only proper ideal of $\frg$, and the quotient $\frg/\ad_{\cC^0}$ is again isomorphic to $\cC^0\cong \ad_{\cC^0}$ (see \cite{AEMN}).

In the split case, $\ad_{\cC_s^0}$ is the ideal of the Chevalley algebra of type $G_2$ generated by the root spaces corresponding to the short roots (see \cite[p.~156]{Steinberg}).

\begin{lemma}\label{le:inner}
Any derivation of $\frg :=\Der(\cC)$ is inner.
\end{lemma}
\begin{proof}
Let $d\in \Der(\frg)$ and let $\fri=\ad_{\cC^0}$ be the unique proper ideal of $\frg$. The ideal $\fri$ is simple, so $\fri=[\fri,\fri]$, and hence $d(\fri)\subseteq [d(\fri),\fri]\subseteq \fri$, so $d\vert_\fri$ is a derivation of $\fri=\ad_{\cC^0}\cong\cC^0$. Since any derivation $\delta$ of $\cC^0$ extends to a derivation of $\cC$ by means of $\delta(1)=0$, it follows that there exists a $\delta\in\frg$ such that $d\vert_\fri=\ad_\delta\vert_\fri$. That is, $d(f)=[\delta,f]$ for any $f\in\fri=\ad_{\cC^0}$, so the derivation $\tilde d=d-\ad_\delta$ satisfies $\tilde d(\fri)=0$.

But if $\tilde d$ is a derivation of $\frg$ such that $\tilde d(\fri)=0$, then $[\tilde d(\frg),\fri]\subseteq \tilde d\bigl([\frg,\fri]\bigr)+[\frg,\tilde d(\fri)]=0$, so $\tilde d(\frg)$ is contained in the centralizer of $\fri$ in $\frg$, which is trivial. In particular, the derivation $\tilde d=d-\ad_\delta$ above is trivial, and hence $d=\ad_\delta$ is inner.
\end{proof}

\begin{theorem}\label{th:char3}
Let $\cC$ be a Cayley algebra over a field $\FF$ of characteristic $3$, and let $\frg :=\Der(\cC)$ be its Lie algebra of derivations. The the adjoint map
\[
\begin{split}
\Ad: \AAut(\cC)&\longrightarrow \AAut(\frg)\\
       f\ &\mapsto\ \varphi(f):d\mapsto fdf^{-1},
\end{split}
\]
is an isomorphism of affine group schemes.
\end{theorem}
\begin{proof}
Let $\FF_{\text{alg}}$ be an algebraic closure of $\FF$. The group homomorphism $\Ad_{\FF_{\text{alg}}}:\Aut(\cC_{\FF_{\text{alg}}})\rightarrow \Der(\cC_{\FF_{\text{alg}}})$ is injective, where $\cC_R=\cC\otimes_\FF R$, for any unital commutative and associative algebra $R$ over $\FF$. This is because $f\ad_xf^{-1}=\ad_{f(x)}$ for any $x\in \cC^0_{\FF_{\text{alg}}}$ and $f\in\Aut(\cC_{\FF_{\text{alg}}})$, so $\Ad_{\FF_{\text{alg}}}(f)=\id$ implies $f\vert_{\cC^0_{\FF_{\text{alg}}}}=\id$; hence, $f=\id$. 

As any $\varphi\in\Aut\left(\frg_{\FF_{\text{alg}}}\right)$ preserves the only proper ideal $\fri_{\FF_{\text{alg}}}=\ad_{\cC^0_{\FF_{\text{alg}}}}$, the automorphism $\varphi$ induces an automorphism $f$ of $\cC^0_{\FF_{\text{alg}}}$, which extends to an automorphism of $\cC_{\FF_{\text{alg}}}$ also denoted by $f$. Then $\varphi\Ad(f)^{-1}\vert_{\fri_{\FF_{\text{alg}}}}=\id$, and, as in the proof above, a simple argument gives $\varphi\Ad(f)^{-1}=\id$, so $\varphi=\Ad(f)$. This shows that $\Ad_{\FF_{\text{alg}}}$ is a bijection.

But it also shows that $\dim \AAut(\frg)=\dim\AAut(\cC)$ and, since this latter group scheme is smooth (this follows from \cite[Proposition 2.2.3]{SV00}, see also \cite[Proof of Theorem 4.35]{EKmon}), 
we obtain $\dim\AAut(\frg)=\dim\Der(\cC)=\dim\Der(\frg)$ by Lemma \ref{le:inner}. Therefore, $\AAut(\frg)$ is smooth.

Since the differential map $\textup{d}(\Ad) : \Der(\cC)\rightarrow\Der(\frg)$, $\delta\mapsto \ad_\delta$, is injective, \cite[(22.5)]{KMRT} shows that $\Ad$ is an isomorphism.
\end{proof}

Denote by $\textup{Isom}(\textup{Cayley})$, $\textup{Isom}(G_2)$, and $\textup{Isom}(\bar A_2)$, the sets of isomorphism classes of Cayley algebras, twisted forms of the Chevalley algebra of type $G_2$, and twisted forms of $\frpsl_3(\FF)$, respectively. Theorem \ref{th:char3} and \eqref{eq:AutCAutC0} immediately give the following consequence, where $[\cA]$ denotes the isomorphism class of the algebra $\cA$.

\begin{corollary}\label{co:char3}
The maps $[\cC]\mapsto[\Der(\cC)]$ and $[\cC]\mapsto [\cC^0]$ give bijections\\ $\textup{Isom}(\textup{Cayley})\rightarrow \textup{Isom}(G_2)$ and $\textup{Isom}(\textup{Cayley})\rightarrow \textup{Isom}(\bar A_2)$, respectively.
\end{corollary}

\medskip


\section{Characteristic $2$}\label{se:char2}

In this section, assume that the characteristic of the ground field $\FF$ is $2$. In \cite[Corollary 4.32]{EKmon} it is proved that the Chevalley algebra of type $G_2$ (i.e., the Lie algebra $\Der(\cC_s)$), is isomorphic to the projective special linear Lie algebra $\frpsl_4(\FF)$. Here we extend this result for the Lie algebra of derivations of any Cayley algebra over $\FF$.

Let $V$ be a finite-dimensional vector space over $\FF$, and let $b$ be a nondegenerate alternating (i.e., $b(u,u)=0$ for any $u \in V$) bilinear form of $V$. Denote by $\frsp(V,b)$ the corresponding symplectic Lie algebra:
\[
\frsp(V,b)=\{ f\in\frgl(V): b(f(u),v)+b(u,f(v))=0\ \forall u,v\in V\}.
\]
In particular, $\frsp_{2n}(\FF)$ denotes the symplectic Lie algebra $\frsp(\FF^{2n},b_s)$, where $b_s$ is the alternating bilinear form with coordinate matrix $\begin{pmatrix} 0&I_n\\ I_n&0\end{pmatrix}$ in the canonical basis. (Notice that, as $\text{char}(\FF) = 2$, there is no need of minus signs.) We identify the elements of $\frsp_{2n}(\FF)$ with their coordinate matrices in the canonical basis.

A matrix in $M_n(\FF)$ is called \emph{alternating} if it has the form $a+a^t$ for some $a\in M_n(\FF)$, where $a^t$ denotes the transpose of $a$. These are the coordinate matrices of the alternating bilinear forms.

\begin{lemma}\label{le:sp62psl4}
The second derived power of the symplectic Lie algebra on a vector space of dimension $6$ is isomorphic to the projective special linear Lie algebra $\frpsl_4(\FF)$:
\[
\frsp_6(\FF)^{(2)}\cong\frpsl_4(\FF).
\]
\end{lemma}
\begin{proof}
For any natural number $n$ we have
\[
\frsp_{2n}(\FF)=\left\{\begin{pmatrix} a&b\\ c&a^t\end{pmatrix} :
a,b,c\in M_n(\FF),\ b^t=b,\, c^t=c\right\}.
\]
A direct computation gives
\[
\begin{split}
\frsp_{2n}(\FF)^{(1)}&=\left\{\begin{pmatrix} a&b\\ c&a^t\end{pmatrix} :
a,b,c\in M_n(\FF),\ \text{$b$ and $c$ alternating}\right\},\\
\frsp_{2n}(\FF)^{(2)}&=\left\{\begin{pmatrix} a&b\\ c&a^t\end{pmatrix} :
a,b,c\in M_n(\FF),\ a\in\frsl_n(\FF),\ \text{$b$ and $c$ alternating}\right\}.
\end{split}
\]
The dimension of $\frsp_{2n}(\FF)^{(2)}$ is then $n^2-1+2\binom{n}{2}=2n^2-n-1$.

Let $V$ be a four-dimensional vector space, and consider the second exterior power $\bigwedge^2V$ as a module for $\frsl(V)$. Fix a nonzero linear isomorphism $\det:\bigwedge^4 V\rightarrow \FF$ and define the nondegenerate alternating bilinear form
\[
\begin{split}
b:\textstyle{\bigwedge^2V\times\bigwedge^2V}&\longrightarrow \FF\\
 (u_1\wedge u_2,v_1\wedge v_2)&\mapsto \det(u_1\wedge u_2\wedge v_1\wedge v_2).
\end{split}
\]
Then, the action of $\frsl(V)$ on $\bigwedge^2V$ gives a Lie algebra homomorphism
\[
\Phi:\frsl(V)\rightarrow \frsp(\textstyle{\bigwedge^2V},b)\cong\frsp_6(\FF),
\]
with kernel $\FF I_V$ (where $I_V$ denotes the identity map on $V$), so $\Phi$ induces an injection $\frpsl(V)\hookrightarrow \frsp_(\bigwedge^2V,b)$. But $\frpsl(V)$ is simple of dimension $14$, so, in particular, $\frpsl(V)^{(2)}=\frpsl(V)$, and the dimension of $\frsp(\bigwedge^2V,b)^{(2)}\cong\frsp_6(\FF)^{(2)}$ is $2\times 3^2-3-1=14$. Therefore, $\Phi$ induces an isomorphism $\frpsl(V)\cong \frsp(\bigwedge^2V,b)^{(2)}$, as required.
\end{proof}

We will need some extra notation. As above, let $(V,b)$ be a finite-dimensional vector space endowed with a nondegenerate alternating bilinear form. Denote by $\frgsp(V,b)$ the Lie algebra of the group of similarities (i.e., the general symplectic Lie algebra):
\begin{multline*}
\frgsp(V,b)=\left\{ f\in\frgl(V): \exists\lambda\in\FF\text{\ such that\ }\right. \\
\left. b(f(u),v)+b(u,f(v))=\lambda b(u,v)\ \forall u,v\in V\right\},
\end{multline*}
and by $\frpgsp(V,b)$ the projective general symplectic Lie algebra (i.e., the quotient of $\frgsp(V,b)$ by the one-dimensional ideal generated by $I_V$). In particular, after choosing a basis, we get the Lie algebras $\frgsp_{2n}(\FF)$ and $\frpgsp_{2n}(\FF)$.

\begin{corollary}\label{co:sp6psl4}
The Lie algebra of derivations of $\frpsl_4(\FF)$ is isomorphic to the projective general symplectic Lie algebra $\frpgsp_6(\FF)$:
\[
\Der\bigl(\frpsl_4(\FF)\bigr)\cong \frpgsp_6(\FF).
\]
\end{corollary}\label{co:pgsp6_Derpsl4}
\begin{proof}
Note that we have the decomposition
\[
\frgsp_6(\FF)=\frsp_6(\FF)\oplus \FF\begin{pmatrix} I_3&0\\ 0&0\end{pmatrix},
\]
and one may easily check that $\frgsp_6(\FF)^{(1)}=\frsp_6(\FF)$.

For any $A\in\frgsp_6(\FF)$, $\ad_A:B\mapsto [A,B]$ leaves invariant $\frgsp_6(\FF)^{(3)}=\frsp_6(\FF)^{(2)}$, which is isomorphic to $\frpsl_4(\FF)$ by Lemma \ref{le:sp62psl4}, so we obtain a Lie algebra homomorphism
\[
\begin{split}
\Phi:\frgsp_6(\FF)&\longrightarrow \Der\bigl(\frsp_6(\FF)^{(2)}\bigr)\, \Bigl(\cong\Der\bigl(\frpsl_4(\FF)\bigr)\,\Bigr),\\
A\ &\mapsto \ \ad_A\vert_{\frsp_6(\FF)^{(2)}}.
\end{split}
\]
The kernel of $\Phi$ is the centralizer in $\frgsp_6(\FF)$ of $\frsp_6(\FF)^{(2)}$, which is $\FF I_6$, so $\Phi$ induces an injection $\frpgsp_6(\FF)\rightarrow \Der\bigl(\frpsl_4(\FF)\bigr)$.  The dimension of $\frpgsp_6(\FF)$ is $21$, and (as it may be calculated in GAP as in \cite{Candido_et_al}) this is also the dimension of $\Der\bigl(\frpsl_4(\FF)\bigr)$. The result follows.
\end{proof}

In fact, in order to prove the previous result, the exact computation of the dimension of $\Der\bigl(\frpsl_4(\FF)\bigr)$ is not required; we just need the bound 
\[
\dim\Der\bigl(\frpsl_4(\FF)\bigr)\leq 21.
\] 
For completeness, let us provide an elementary proof of this fact.

\begin{lemma}
$\dim\Der\bigl(\frpsl_4(\FF)\bigr)\leq 21$.
\end{lemma}
\begin{proof}
For $1\leq i,j\leq 4$, let $E_{ij}$ be the matrix in $\frgl_4(\FF)$ with $1$ in the $(i,j)$ entry and $0$'s elsewhere. Denote by $\bar A$ the class of a matrix $A\in\frsl_4(\FF)$ in $\frpsl_4(\FF)$.

Then, $\frsl_4(\FF)$ is graded by $\ZZ^3$, with
\begin{gather*}
\degree(E_{12})=(1,0,0)=-\degree(E_{21}),\\
\degree(E_{23})=(0,1,0)=-\degree(E_{32}),\\
\degree(E_{34})=(0,0,1)=-\degree(E_{43}).
\end{gather*}
As the identity matrix $I_4$ is homogeneous of degree $(0,0,0)$, this induces a grading by $\ZZ^3$ on $\frg=\frpsl_4(\FF)$:
\[
\Gamma:\frg=\bigoplus_{\alpha\in\ZZ^3}\frg_\alpha,
\]
with support 
\[
\begin{split}
\supp\Gamma&\bydef \{\alpha\in\ZZ^3: \frg_\alpha\ne 0\}\\
 &=\{(0,0,0),\pm(1,0,0),\pm(0,1,0),\pm(0,0,1),\pm(1,1,0),\pm(0,1,1),\pm(1,1,1)\}.
\end{split}
\]
Let $\frh=\espan{H_1\bydef\overline{E_{11}+E_{22}},H_2\bydef\overline{E_{22}+E_{33}}}=\frg_{(0,0,0)}$ be the `diagonal' subalgebra of $\frg$. Then, the decomposition in eigenspaces for the adjoint action of $\frh$ is $\frg=\frh\oplus\frg_1\oplus\frg_2\oplus\frg_3$, with
\[
\begin{split}
\frg_1&=\espan{\bar E_{12},\bar E_{21},\bar E_{34},\bar E_{43}}
 =\frg_{(1,0,0)}\oplus\frg_{(-1,0,0)}\oplus\frg_{(0,0,1)}\oplus\frg_{(0,0,-1)},\\
\frg_2&=\espan{\bar E_{23},\bar E_{32},\bar E_{14},\bar E_{41}}
 =\frg_{(0,1,0)}\oplus\frg_{(0,-1,0)}\oplus\frg_{(1,1,1)}\oplus\frg_{(-1,-1,-1)},\\
\frg_3&=\espan{\bar E_{13},\bar E_{31},\bar E_{24},\bar E_{42}}
 =\frg_{(1,1,0)}\oplus\frg_{(-1,-1,0)}\oplus\frg_{(0,1,1)}\oplus\frg_{(0,-1,-1)}.
\end{split}
\]
As $\frg$ is $\ZZ^3$-graded, so is $\Der(\frg)$. Several steps are required now:

\begin{romanenumerate}
\item $\dim\Der(\frg)_{(0,0,0)}\leq 3$, and $d(\frh)=0$ for any $d\in \Der(\frg)_{(0,0,0)}$.

\begin{proof}
Any $d\in \Der(\frg)_{(0,0,0)}$ preserves the one-dimensional spaces $\frg_\alpha$, for $\alpha\in\supp\Gamma\setminus\{(0,0,0)\}$. Then $d$ and $\ad_\frh$ commute, so $d(\frh)=0$. Also $d(\bar E_{12})=\lambda \bar E_{12}$ and $d(\bar E_{23})=\mu \bar E_{23}$ for some $\lambda,\mu\in\FF$. From $d(\frh)=0$, we obtain that $d(\bar E_{21})=-\lambda \bar E_{21}$ and $d(\bar E_{32})=-\mu\bar E_{32}$. Hence $\tilde d\bydef d-\ad_{\mu H_2+\lambda H_1}$ annihilates $\bar E_{12}$ and $\bar E_{23}$ (and $\bar E_{21}$ and $\bar E_{32}$). Since the elements $\bar E_{12}$, $\bar E_{21}$, $\bar E_{23}$, $\bar E_{32}$, $\bar E_{34}$ and $\bar E_{43}$ generate $\frg$, it follows that $\tilde d$ is determined by the value $\tilde d(E_{34})$. We conclude that $\dim\Der(\frg)_{(0,0,0)}-\dim\ad_\frh\leq 1$.
\end{proof}

\item $\Der(\frg)=\ad_\frg +\{d\in\Der(\frg): d(\frh)=0\}$.

\begin{proof}
We already have $\Der(\frg)_{(0,0,0)}\subseteq \{d\in\Der(\frg):d(\frh)=0\}$, and it is clear that $d\in\Der(\frg)_\alpha$, with $\alpha \in \ZZ^3 \setminus \supp\Gamma$, implies $d(\frh)\subseteq \frg_\alpha=0$. On the other hand, if $d\in \Der(\frg)_\alpha$, with $\alpha \in\supp\Gamma\setminus \{(0,0,0)\}$, then the restriction of $d$ to $\frh$ defines a linear map $\beta :\frh\rightarrow \FF$ by $d(H)=\beta(H)\bar E_{rs}$, where $\bar E_{rs}$ is the basic element in the one-dimensional space $\frg_\alpha$. But for any $H,H'\in\frh$ we get:
\[
\begin{split}
0=d([H,H'])&=[d(H),H']+[H,d(H')]\\
 &= \beta(H)[\bar E_{rs},H']+ \beta(H')[H,\bar E_{rs}]\\
 &=\bigl(-\beta(H)\gamma(H')+\beta(H')\gamma(H)\bigr)\bar E_{rs},
\end{split}
\]
where $\gamma:\frh\rightarrow \FF$ is the nonzero linear form such that $[H,\bar E_{rs}]=\gamma(H)\bar E_{rs}$ for any $H\in\frh$. Hence $\beta$ is a scalar multiple of $\gamma$, and if $\beta=\nu\gamma$ with $\nu\in \FF$, then $(d-\nu\ad_{\bar E_{rs}})(\frh)=0$. (This argument is similar to the one in \cite[Proposition 8.1]{EK12}.)
\end{proof}

\item Now note that $\bar E_{12}$, $\bar E_{23}$, $\bar E_{34}$ and $\bar E_{41}$ generate $\frg$. Let $0\ne d\in \Der(\frg)_\alpha$, with $\alpha \in \mathbb{Z}_3 \setminus \{(0,0,0)\}$, and $d(\frh)=0$. Then $d(\frg_i)\subseteq \frg_i$, for $i=1,2,3$, and
\[
\begin{split}
&d(\bar E_{12})\in\frg_1\ 
\text{so either $d(\bar E_{12})=0$ or $\alpha\in\{(-2,0,0),(-1,0,1),(-1,0,-1)\}$,}\\
&d(\bar E_{23})\in\frg_2\ 
\text{so either $d(\bar E_{23})=0$ or $\alpha \in\{(0,-2,0),(1,0,1),(-1,-2,-1)\}$,}\\
&d(\bar E_{34})\in\frg_1\ 
\text{so either $d(\bar E_{34})=0$ or $\alpha\in\{(0,0,-2),(1,0,-1),(-1,0,1)\}$,}\\
&d(\bar E_{41})\in\frg_2\ 
\text{so either $d(\bar E_{41})=0$ or $\alpha\in\{(2,2,2),(1,2,1),(1,0,-1)\}$.}
\end{split}
\]
But $\Der(\frg)_\alpha=0$ for $\alpha\in\{(-2,0,0),(0,-2,0),(0,0,-2),(2,2,2)\}$, because any $d$ in one of these homogeneous spaces annihilates the generators $\bar E_{21}$, $\bar E_{32}$, $\bar E_{43}$ and $\bar E_{14}$. Hence, our homogeneous $d$ belongs to $\Der(\frg)_\alpha$ with $\alpha \in X\bydef\{\pm(1,0,1),\pm(1,0,-1),\pm(1,2,1)\}$. Now if, for instance, $d\in\Der(\frg)_{(-1,0,-1)}$, then $d$ annihilates $\bar E_{23}\in \frg_{(0,1,0)}$ (because $(0,1,0)+(-1,0,-1)\not\in\supp\Gamma$),  $\bar E_{41}\in\frg_{(-1,-1,-1)}$, and $\bar E_{24}\in\frg_{(1,1,0)}$. Moreover, since $\bar E_{34}=[[\bar E_{31},\bar E_{12}],\bar E_{24}]$, it follows that $d$ is determined by the value $d(\bar E_{12})$. Thus, we get $\dim\Der(\frg)_{(-1,0,-1)}\leq 1$. A similar argument applies to the other possibilities, so $\dim\Der(\frg)_\alpha\leq 1$ for any $\alpha\in X$.

\item Therefore, we conclude that
\[
\dim\Der(\frg)-\dim\frg=\Bigl(\dim\Der(\frg)_{(0,0,0)}-\dim\frh\Bigr)+\hspace{-5pt}\sum_{\alpha \in X}\hspace{-5pt}\dim\Der(\frg)_\alpha \leq 7,
\]
so that $\dim\Der(\frg)\leq \dim\frg+7=21$. \qedhere
\end{romanenumerate}
\end{proof}

\smallskip

Our next result extends \cite[Corollary 4.32]{EKmon}.

\begin{theorem}\label{th:DerC_psl4}
Let $\cC$ be a Cayley algebra over a field $\FF$ of characteristic $2$. The Lie algebra of derivations $\Der(\cC)$ is isomorphic to the projective special linear Lie algebra $\frpsl_4(\FF)$. (Independently of the isomorphism class of $\cC$!)
\end{theorem}
\begin{proof}
Recall that $\Der(\cC)$ is a $14$-dimensional simple Lie algebra, a twisted form of $\Der(\cC_s)$, which is the Chevalley algebra of type $G_2$.

Any $d\in\Der(\cC)$ leaves the norm $q$ invariant, annihilates the unity $1$ and preserves $\cC^0$, the subspace of trace zero elements. Since the characteristic is $2$, the unity $1$ is in $\cC^0$, so $d$ induces an element $\tilde d$ in the symplectic Lie algebra $\frsp\bigl(\cC^0/\FF 1,\tilde b_q\bigr)\cong\frsp_6(\FF)$, where $\tilde b_q$ is the nondegenerate alternating bilinear form on $\cC^0$ induced by $b_q$. (Note that $\cC^0=\{x\in\cC: b_q(x,1)=0\}$, so $\tilde b_q$ is nondegenerate.)

Therefore, we have a homomorphism of Lie algebras
\[
\begin{split}
\Phi:\Der(\cC)&\longrightarrow \frsp\bigl(\cC^0/\FF 1,\tilde b_q\bigr),\\
 d\quad &\mapsto\quad \tilde d.
\end{split}
\]
The simplicity of $\Der(\cC)$ implies that $\Phi$ is injective, and hence 
\[
\Phi\bigl(\Der(\cC)\bigr)=\Phi\bigl(\Der(\cC)^{(2)}\bigr)\subseteq 
 \frsp\bigl(\cC^0/\FF 1,\tilde b_q\bigr)^{(2)}\cong\frsp_6(\FF)^{(2)}.
\]
By dimension count, the image of $\Phi$ is $\frsp\bigl(\cC^0/\FF 1,\tilde b_q\bigr)^{(2)}$, which is isomorphic to $\frsp_6(\FF)^{(2)}$, and hence to $\frpsl_4(\FF)$ by Lemma \ref{le:sp62psl4}.
\end{proof}

\smallskip

The previous theorem shows that, in characteristic $2$, it is no longer true that two Cayley algebras are isomorphic if and only if their Lie algebras of derivations are isomorphic.

\begin{remark}
Given an irreducible root system of type $X_r$ and its corresponding Chevalley algebra $\frg$ over a field $\FF$, the quotient $\frg/Z(\frg)$ (where $Z(\frg)$ is the center of $\frg$) is usually called the \emph{classical Lie algebra of type $X_r$}. In particular, in characteristic $2$, Theorem \ref{th:DerC_psl4} implies that the classical Lie algebras of type $A_3$ and $G_2$ coincide.
\end{remark}

Write $\cA :=M_6(\FF)$, and let $\sigma$ be the symplectic involution (attached to the standard alternating form $b_s$), such that, for any $X\in \cA$, $\sigma(X)$ is the adjoint relative to $b_s$.

\begin{theorem}\label{th:AAutDerC}
The affine group scheme of automorphisms of $\Der(\cC_s)$ is isomorphic to the affine group scheme of automorphisms of the algebra with involution $(\cA,\sigma)$.
\end{theorem}

Over an algebraic closure $\FF_{\text{alg}}$, the group $\Aut(\frpsl_4(\FF_{\text{alg}})$ is the adjoint Chevalley group of type $C_3$ (see \cite{HogewejII}), and hence it is isomorphic to projective general symplectic group $\textrm{PGSp}_6(\FF_{\text{alg}})\cong\Aut\bigl(\cA,\sigma\bigr)$. However, Theorem \ref{th:AAutDerC} considers arbitrary fields, and we shall give an explicit isomorphism of schemes in its proof. Note that this result over $\FF_{\text{alg}}$ shows that $\AAut(\frpsl_4(\FF))$ is connected, and Corollary \ref{co:pgsp6_Derpsl4} shows that it is smooth.

\begin{proof}[Proof of Theorem \ref{th:AAutDerC}]
As in the proof of Theorem \ref{th:DerC_psl4}, we may identify $\Der(\cC_s)\cong\frpsl_4(\FF)$ with the Lie algebra $\Skew(\cA,\sigma)^{(2)}$, where $\Skew(\cA,\sigma)\bydef \{x\in\cA: \sigma(x)=x\}$ (as the characteristic is two!) Consider the morphism of affine group schemes
\[
\begin{split}
\varphi:\AAut(\cA,\sigma)&\longrightarrow \AAut\Bigl(\Skew(\cA,\sigma)^{(2)}\Bigr),\\
f:\cA_R\rightarrow\cA_R&\mapsto f\vert_{\Skew(\cA_R,\sigma_R)^{(2)}},
\end{split}
\]
where $R$ is a unital, commutative and associative $\FF$-algebra, $\cA_R=\cA\otimes_\FF R$, and $\sigma_R=\sigma\otimes\id$ is the induced involution in $\cA_R$, so $\Skew(\cA_R,\sigma_R)^{(2)}=\Skew(\cA,\sigma)^{(2)}\otimes_\FF R$.

The group homomorphism on points in an algebraic closure $\FF_{\text{alg}}$
\[
\varphi_{\FF_{\text{alg}}}:\Aut(\cA_{\FF_{\text{alg}}})\rightarrow \Aut\bigl(\Skew(\cA_{\FF_{\text{alg}}},\sigma_{\FF_{\text{alg}}})^{(2)}\bigr)
\]
is injective, because $\Skew(\cA_{\FF_{\text{alg}}},\sigma_{\FF_{\text{alg}}})^{(2)}= \frsp_6(\FF_{\text{alg}})^{(2)}$ generates $\cA_{\FF_{\text{alg}}}=M_6(\FF_{\text{alg}})$ as an algebra. Also, the differential $\text{d}\varphi$ is an isomorphism: it is the isomorphism $\Phi$ in the proof of Corollary \ref{co:pgsp6_Derpsl4}. Therefore, $\varphi$ is a closed imbedding (\cite[(22.2)]{KMRT}). But both schemes are smooth, connected, and of the same dimension, so $\varphi$ is an isomorphism.
\end{proof}

Since $\AAut(\frpsl_4(\FF)$ is smooth, the set of isomorphism classes of twisted forms of $\frpsl_4(\FF)$ is in bijection with $H^1(\FF,\AAut(\frpsl_4(\FF))$ (see \cite[Chapters 17 and 18]{Waterhouse}), and also the set of isomorphism clases of central simple associative algebras of degree $6$ endowed with a symplectic involution is in bijection with $H^1(\FF,\AAut(M_6(\FF),\sigma)=H^1(\FF,\textrm{PGSp}_6(\FF))$. Our last result is then a direct consequence of Theorem \ref{th:AAutDerC}.

\begin{corollary}
Let $\FF$ be a field of characteristic $2$. The map that sends any central simple associative algebra of degree $6$ over $\FF$ endowed with a symplectic involution $(\cB,\tau)$ to the Lie algebra $\Skew(\cB,\tau)^{(2)}$ gives a bijection between the set of isomorphism classes of such pairs $(\cB,\tau)$ to the set of twisted forms over $\FF$ of the Lie algebra $\frpsl_4(\FF)$.
\end{corollary}

Recall that $\frpsl_4(\FF)$ is (isomorphic to) the Chevalley algebra of type $G_2$, so this corollary gives the twisted forms of the classical simple Lie algebras of type $G_2$ in characteristic $2$.

\bigskip 


\end{document}